\documentclass[a4paper,10pt]{amsart}
\usepackage{amsfonts, amsmath, amsthm, amssymb, bbm}

\usepackage[T2A]{fontenc}
\usepackage[cp1251]{inputenc}
\usepackage[russian, english]{babel}
\usepackage{xcolor,graphics}

\newtheorem{theorem}{Theorem}[section]

\newtheorem{lemma}[theorem]{Lemma}

\newcommand{\E}{\mathbb{E}}
\renewcommand{\P}{\mathbb{P}}
\newcommand{\R}{\mathbb{R}}
\renewcommand{\C}{\mathbb{C}}
\renewcommand{\Re}{\operatorname{Re}}
\renewcommand{\Im}{\operatorname{Im}}
\newcommand{\bx}{\mathbf{x}}

\newcommand{\bs}{\mathbf{s}}

\newcommand{\by}{\mathbf{y}}
\newcommand{\bz}{\mathbf{z}}
\newcommand{\bu}{\mathbf{u}}

\newcommand{\be}{\boldsymbol{\eta}}
\newcommand{\dd}{{\rm d}}
\newcommand{\rv}{{\rm v}}
\newcommand{\ind}{\mathbbm{1}}

\title[Correlations between real and complex zeros]{Correlations between real and complex zeros \\ of a random polynomial}
\thanks{The work  was done with the financial support of the Bielefeld University (Germany) in terms of project SFB 701. The work of the third author is supported by the grant RFBR
16-01-00367 and by the Program of Fundamental
Researches of Russian Academy of Sciences ``Modern Problems of
Fundamental Mathematics''.}

\author[F.~G\"otze]{Friedrich G\"otze}
 \address{Friedrich G\"otze, Faculty of Mathematics,
 Bielefeld University,
 P. O. Box 10 01 31,
 33501 Bielefeld, Germany}
 \email{goetze@math.uni-bielefeld.de}

\author[D.~Koleda]{Denis Koleda}
 \address{Denis Koleda, Institute of Mathematics, National Academy of Sciences of Belarus, 220072 Minsk, Belarus}
 \email{koledad@rambler.ru}

\author[D.~Zaporozhets]{Dmitry Zaporozhets}
 \address{Dmitry Zaporozhets\\
 St.\ Petersburg Department of
 Steklov Institute of Mathematics,
 Fontanka~27,
  191011 St.\ Petersburg,
 Russia}
 \email{zap1979@gmail.com}

\begin{document}

\begin{abstract}
Consider a random polynomial
$$
G(z):=\xi_0+\xi_1z+\dots+\xi_nz^n,\quad z\in\C,
$$
where $\xi_0,\xi_1,\dots,\xi_{n}$ are independent real-valued random variables with probability density functions $f_0,\dots,f_n$. We give an explicit formula for the  mixed $(k,l)$-correlation function $\rho_{k,l}:\R^k\times\C_+^l \to\R_+$ between $k$ real and $l$ complex zeros of $G_n$.
\end{abstract}

\maketitle

\section{Introduction}

Let $\xi_0,\xi_1,\dots,\xi_{n}$ be independent \emph{real-valued} random variables with probability density functions $f_0,\dots,f_n$. Consider a random polynomial
\begin{equation}\label{1914}
G(z):=\xi_0+\xi_1z+\dots+\xi_nz^n,\quad z\in\C.
\end{equation}
With probability one, all zeros of $G$ are simple. Denote by $\mu$ the empirical measure counting zeros of $G$:
\[
\mu:=\sum_{z:G(z)=0}\delta_z,
\]
where $\delta_z$ is the unit point mass at $z$. The random measure $\mu$ may be regarded as a  random point process on $\mathbb{C}$. The natural way of describing the distribution of a point process is via its $\emph correlation functions$. However, since the coefficients of $G$ are real, its  zeros are symmetric with respect to the real line, and some of them are real. Therefore, the natural configuration space for the point process $\mu$ must be a ``separated'' union $\C_+\cup\R$ with topology induced by the union of topologies in $\C_+$ and $\R$. Instead of considering the correlation functions of the process on $\C_+\cup\R$, the equivalent way is to investigate the \emph{mixed $(k,l)$-correlation functions} (see~\cite{TV15}).
We call functions   $\rho_{k,l}\,:\,\R^k\times\mathbb{C}_+^l \to\R_+$, where $k+2l\le n$, the mixed $(k,l)$-correlation functions of zeros of $G$, if for any family of mutually disjoint Borel subsets $B_1,\dots,B_k\subset\R$ and $B_{k+1},\dots,B_{k+l}\subset\C_+$,
\begin{equation}\label{eq-rhokl}
\E\,\left[\prod_{i=1}^{k+l}\mu(B_i)\right]=\int_{B_1}\dots\int_{B_{k+l}}\,\rho_{k,l}(\bx,\bz)\,\dd\bx\,\dd\bz.
\end{equation}
Here and subsequently, we write
$$
\bx:=(x_1,\dots,x_k)\in\R^k,\quad \bz:=(z_1,\dots,z_l)\in\C_+^l.
$$

\bigskip

The most popular class of random polynomilas are \emph{Kac polynomials}, when $\xi_i$'s are i.i.d.  Sometimes the i.i.d. coefficients are considered with some non-random weights $c_i$'s. The common examples are \emph{flat or Weil polynomials} ($c_i=\sqrt{1/i!}$) and \emph{elliptic polynomials} ($c_i=\sqrt{n\choose k}$).

The $(1,0)$-correlation function $\rho_{1,0}$ is called a density of real zeros. Being integrated over $\R$, it equals the average number of real zeros of $G$. The asymptotic properties of this object as $n\to\infty$ have been intensively  studied for many years,  mostly for Kac polynomials; see the historical background  in~\cite{BS86} and the survey of the most recent results in~\cite{TV15}. We just mention some  works:~\cite{mK43}, \cite{EO56}, \cite{dS69}, \cite{IM71-1}, \cite{NZ09}, \cite{LPX15}, \cite{kS16b}.

In the same manner, $\rho_{0,1}$ is called a density of complex zeros, that is an expectation of the empirical measure $\mu$ counting non-real zeros. Its limit behavior as $n\to\infty$ is also of a great interest, see~\cite{SS62}, \cite{IZ97}, \cite{IZ13}, \cite{KZ14b}, \cite{iP16}, \cite{PR16},   and the references given there.

At the same time, there are relatively few papers on higher-order correlation functions of zeros. The most known results are due to Bleher and Di~\cite{BD97}, \cite{BD04} who studied the correlations between real zeros for elliptic and Kac polynomials, and to Tao and Vu~\cite{TV15} who under some moment conditions on $\xi_i$ proved the asymptotic universality for the mixed correlation functions for elliptic, Weil, and Kac polynomials.

\bigskip

The aim of this note is to derive the explicit formula for the mixed correlation functions of zeros of random polynomials with arbitrary absolutely continuous random coefficients. To simplify calculations, we assume that $\xi_i$'s are independent, however, our proof works for the coefficients having an arbitrary joint probability density function.

\section{Main result}

Recall that $f_0,\dots,f_n$ are the probability density functions of the coefficients $\xi_0,\dots,\xi_n$ of $G$; see~\eqref{1914}.  For $m=1,\dots,n,$ consider a function $\rho_m:\C^m\to\R$ defined as
\begin{multline}\label{eq-rho-b}
\rho_m(z_1,\dots,z_m):= \prod_{1\le i < j \le m} |z_i - z_j|\\
\times\int\limits_{\mathbb{R}^{n-m+1}} \prod_{i=0}^{n}f_i\left(\sum_{j=0}^{n-m}(-1)^{m-i+j}\sigma_{m-i+j}(z_1,\dots,z_m)t_j\right) \prod_{i=1}^m \Bigg|\sum_{j=0}^{n-m} t_j z_i^j\Bigg|\dd t_0\dots \dd t_{n-m},
\end{multline}
where we used the following notation  for the elementary symmetric polynomials:
\begin{equation}\label{1738}
\sigma_i(z_1,\dots,z_m) :=
\left\{
  \begin{array}{ll}
    1, & \hbox{if}\quad i=0,\\
    \sum_{1\le j_1 < \dots < j_i\le m} z_{j_1} z_{j_2} \dots z_{j_i}, & \hbox{if}\quad 1\leq i\le m, \\
    0, & \hbox{otherwise.}
  \end{array}
\right.
\end{equation}
It is tacitly assumed that the arguments of $f_i$'s are well-defined: we will always consider only the collections of points $(z_1,\dots,z_m)$ such that all symmetric functions of them are real.

We will also need the notation for the absolute value of the  Vandermonde determinant:
\begin{equation}\label{1330}
  \rv_m(z_1,\dots,z_m):=\prod_{1\le i < j \le m} |z_i - z_j|.
\end{equation}

It was proved in~\cite{GKZ15b} that the correlation functions of real zeros are given by~\eqref{eq-rho-b}:
$$
\rho_{k,0}(\bx)=\rho_{k}(\bx)
$$
for all $\bx\in\R^k$. The following theorem generalizes this relation to all mixed $(k,l)$-correlation functions.

\begin{theorem}\label{thm-corr-rel}
For all $(\bx,\bz)\in\R^k\times\C_+^l$,
\begin{equation}\label{eq-corr-rel}
\rho_{k,l}(\bx,\bz) = 2^l \rho_{k+2l}(\bx,\bz,\bar\bz),
\end{equation}
where $\rho_{k+2l}$ is defined in~\eqref{eq-rho-b}.
\end{theorem}

It is interesting to note that essentially the correlations between real zeros and the correlations between complex zeros are given by the same function $\rho_m$. In particular, $\rho_2$ provides the  formula for the density of complex zeros as well as for the two-point correlation function of real zeros:
$$
\rho_{0,1}(z)=2\rho_2(z,\bar z),\quad \rho_{2,0}(x,y)=\rho_2(x,y).
$$
The latter formula (with different notations) was obtained in~\cite{dZ05}.

The proof of Theorem~\ref{thm-corr-rel} is be given in Section~\ref{1245}. In the next section we derive some corollaries.

\section{Applications of Theorem~\ref{thm-corr-rel}}

Substituting~\eqref{eq-rho-b} into~\eqref{eq-corr-rel} gives
\begin{multline*}
  \rho_{k,l}(\bx,\bz) =2^l\rv_{k+2l}(\bx,\bz,\bar\bz)\int\limits_{\mathbb{R}^{n-k-2l+1}} \prod_{i=0}^{n}f_i\left(\sum_{j=0}^{n-k-2l}(-1)^{k-i+j}t_j\sigma_{k+2l-i+j}(\bx,\bz,\bar\bz)\right)  \\
    \times \prod_{i=1}^k \Bigg|\sum_{j=0}^{n-k-2l} t_j x_i^j\Bigg|\cdot \prod_{i=1}^l \Bigg|\sum_{j=0}^{n-k-2l} t_j z_i^j\Bigg|^2\dd t_0\dots \dd t_{n-k-2l}.
\end{multline*}
In particular, taking $k=1,l=0$  yields the formula for the density of real zeros:
\begin{align*}
  \rho_{1,0}(x) =\int\limits_{\mathbb{R}^{n}} \prod_{i=0}^{n}f_i\left(t_{i-1}-xt_i\right)\Bigg|\sum_{j=0}^{n-1} t_j x^j\Bigg|\dd t_0\dots \dd t_{n-1},
\end{align*}
where we set $t_{-1}=t_n:=0$.

Taking $k=0,l=1$ yields the formula for the density of complex zeros:
\begin{align*}
  \rho_{0,1}(z)=4|\Im z|\int\limits_{\mathbb{R}^{n-1}} \prod_{i=0}^{n}f_i(t_{i-2}-2t_{i-1}\Re z+t_i|z|^2)
    \Bigg|\sum_{j=0}^{n-2} t_j z^j\Bigg|^2\dd t_0\dots \dd t_{n-2},
\end{align*}
where we set $t_{-2}=t_{-1}=t_{n-1}=t_n:=0$.

Taking $k=n-2l$ we obtain the (non-normalized) joint density of all zeros given that $G$ has exactly $n-2l$ real zeros:
\begin{align}\label{1638}
  \rho_{n-2l,l}(\bx,\bz)=2^l\rv_{n}(\bx,\bz,\bar\bz)\int\limits_{\R }|t|^n \prod_{i=0}^{n}f_i\left((-1)^{n-i}t\sigma_{n-i}(\bx,\bz,\bar\bz)\right)\dd t.
\end{align}

It is easy to derive {from~\eqref{eq-rhokl}} that
\begin{align*}
  \E\left[\frac{\mu(\R)!}{(\mu(\R)-n+2l)!}\frac{\mu(\C_+)!}{(\mu(\C_+)-l)!}\right]=\int_{\R^{n-2l}}\int_{\C^{l}_+}\rho_{n-2l,l}(\bx,\bz)\dd\bx\dd\bz,
\end{align*}
where we used the convention {$0!:=1$ and $q!:=\infty$ for any integer $q < 0$}.
Since with probability one $\mu(\R)+2\mu(\C_+)=n$, the random variable under expectation is non-zero if and only if $\mu(\R)=n-2l$. Thus we obtain the probability that $G$  has exactly $n-2l$ real zeros:
\begin{multline*}
 \P[\mu(\R)=n-2l]=\frac{2^l}{l!(n-2l)!}\int_{\R^{n-2l}}\int_{\C^{l}_+}\rv_{n}(\bx,\bz,\bar\bz)\\\times\int\limits_{\R }|t|^n \prod_{i=0}^{n}f_i\left((-1)^{n-i}t\sigma_{n-i}(\bx,\bz,\bar\bz)\right)\dd t\dd\bx\dd\bz.
\end{multline*}
This formula was obtained in~\cite{dZ04}. Now we consider some examples.

{\bf Uniformly distributed coefficients.}
In algebraic number theory, random polynomials with independent and uniformly distributed on $[-1,1]$ coefficients are of a special interest (see~\cite{dK14}, \cite{fGdZ14}, \cite{GKZ15}). Let us apply Theorem~\ref{thm-corr-rel} to this case.

Suppose that
$$
f_i=\frac12\mathbbm{1}[-1,1],\quad i=0,\dots,n.
$$
Then it follows from Theorem~\ref{thm-corr-rel} that
\begin{multline*}
\rho_{k,l}(\mathbf{x}) = 2^{l-n-1}\rv_{k+2l}(\bx,\bz,\bar\bz)\\\times\int\limits_{D_{\bx,\bz}}\prod_{i=1}^k \Bigg|\sum_{j=0}^{n-k-2l} t_j x_i^j\Bigg|\cdot \prod_{i=1}^l \Bigg|\sum_{j=0}^{n-k-2l} t_j z_i^j\Bigg|^2\dd t_0\dots \dd t_{n-k-2l},
\end{multline*}
where the domain of integration $D_{\bx,\bz}$ is defined as
\begin{align*}
D_{\bx,\bz}:=\Bigg\{(t_0,\dots,t_{n-k-2l})\in\R^{n-k-2l+1}:\max_{0\leq i\leq n}\Bigg|\sum_{j=0}^{n-k-2l}(-1)^{k-i+j}t_j\sigma_{k+2l-i+j}(\bx,\bz,\bar\bz)\Bigg|\leq1
\Bigg\}.
\end{align*}
In particular,
$$
\rho_{n-2l,l}(\bx,\bz) =\frac{2^{l-n}}{n+1}\cdot\frac{\rv_n(\bx,\bz,\bar\bz)}{(\max_{0\leq i\leq n}|\sigma_i(\bx,\bz,\bar\bz)|)^{n+1}}.
$$

{\bf Gaussian distribution.}
Suppose that
\[
f_i(t) = \frac{1}{\sqrt{2\pi}v_i } \exp\left(- \frac{t^2}{2v_i^2}\right),\quad i=0,\dots,n.
\]
Using the formula for the $n$-th absolute moment of the Gaussian distribution, we derive from~\eqref{1638} that
$$
\rho_{n-2l,l}(\bx,\bz) =\frac{2^{l+1/2} \Gamma\left(\frac{n+1}{2}\right)}{(2\pi)^{n/2} v_0\dots v_n}
 \left(\sum_{i=0}^n\frac{\sigma_{n-i}^2(\bx,\bz,\bar\bz)}{v_i^2}\right)^{-(n+1)/2}\rv_{n}(\bx,\bz,\bar\bz).
$$
In particular, for i.i.d coefficients we have
$$
\rho_{n-2l,l}(\bx,\bz)  =\frac{2^{l+1/2} \Gamma\left(\frac{n+1}{2}\right)}{(2\pi)^{n/2} }
 \left(\sum_{i=0}^n\sigma_{n-i}^2(\bx,\bz,\bar\bz)\right)^{-(n+1)/2}\rv_{n}(\bx,\bz,\bar\bz).
$$


{\bf Exponential distribution.}
Suppose that
\[
f_i(t) =\exp(-t )\mathbbm{1}\{t\geq0\},\quad i=0,\dots,n.
\]
By an elementary integration,~\eqref{1638} implies
$$
\rho_{n-2l,l}(\bx,\bz) = \left(\sum_{i=0}^n (-1)^i\sigma_i(\bx,\bz,\bar\bz)\right)^{-(n+1)}\rv_n(\bx,\bz,\bar\bz)\ind{\{(-1)^{i}\sigma_{i}(\bx,\bz,\bar\bz)\geq0\}}.
$$
Using the trivial identity
$$
\sum_{i=0}^n (-1)^i\sigma_i(w_1,\dots,w_n)=\prod_{i=1}^n(1-w_i),
$$
we obtain
$$
\rho_{n-2l,l}(\bx,\bz) = \frac{n!\,\rv_n(\bx,\bz,\bar\bz) \, \ind{\{(-1)^{i}\sigma_{i}(\bx,\bz,\bar\bz)\geq0\}}}{\left(\prod_{i=1}^{n-2l}(1-x_i)\prod_{i=1}^{l}(1-2\Re z_i+|z_i|^2)\right)^{n+1}}.
$$

Random polynomials with  i.i.d. exponential coefficients have been investigated  in~\cite{wL11}, \cite{GZ13}.

%
\section{Proof of Theorem~\ref{thm-corr-rel}}\label{1245}

\subsection{{Preliminaries}}

Suppose that $x_1,\dots,x_k\in\R$ and $z_1,\dots,z_l\in\C_+$ are different zeros of $G$.
It means that
\begin{equation}\label{2018}
\begin{pmatrix}
1 & x_1 & \dots & x_1^{n} \\
\vdots & \vdots & \ddots & \vdots \\
1 & x_k & \dots & x_k^{n} \\
1 & \Re z_1 & \dots & \Re z_1^{n} \\
0 & \Im z_1 & \dots & \Im z_1^{n} \\
\vdots & \vdots & \ddots & \vdots \\
1 & \Re z_l & \dots & \Re z_l^{n} \\
0 & \Im z_l & \dots & \Im z_l^{n} \\
\end{pmatrix}
\begin{pmatrix}
\xi_0\\
\vdots\\
\xi_n
\end{pmatrix}
=\mathbf{0}.
\end{equation}

Denote by $V(\bx,\bz)$ {the real Vandermonde type} matrix
$$
V(\bx,\bz):=
\begin{pmatrix}
1 & x_1 & \dots & x_1^{k+2l} \\
\vdots & \vdots & \ddots & \vdots \\
1 & x_k & \dots & x_k^{k+2l} \\
1 & \Re z_1 & \dots & \Re z_1^{k+2l} \\
0 & \Im z_1 & \dots & \Im z_1^{k+2l} \\
\vdots & \vdots & \ddots & \vdots \\
1 & \Re z_l & \dots & \Re z_l^{k+2l} \\
0 & \Im z_l & \dots & \Im z_l^{k+2l} \\
\end{pmatrix}.
$$
Then,~\eqref{2018} is equivalent to
\begin{equation}\label{1400}
\begin{pmatrix}
\sum_{j=k+2l}^n\xi_jx_1^j\\
\vdots\\
\sum_{j=k+2l}^n\xi_jx_k^j\\
\Re\sum_{j=k+2l}^n\xi_jz_1^j\\
\Im\sum_{j=k+2l}^n\xi_jz_1^j\\
\vdots\\
\Re\sum_{j=k+2l}^n\xi_jz_l^j\\
\Im\sum_{j=k+2l}^n\xi_jz_l^j\\
\end{pmatrix}
=-V(\bx,\bz)
\begin{pmatrix}
\xi_{0}\\
\vdots\\
\xi_{k+2l-1}
\end{pmatrix}.
\end{equation}
It is easy to check that $V(\bx,\bz)$ satisfies
\begin{equation}\label{1931}
|\det V(\bx,\bz)|=2^{-l}\rv_{k+2l}(\bx,\bz,\bar\bz),
\end{equation}
where $\rv_{k+2l}$ is defined in~\eqref{1330}.

Consider a random function $\be=(\eta_0,\dots,\eta_{k+2l-1})^T:\R^{k}\times\C^l\to\R^{k+2l}$ defined as
\begin{equation}\label{2327}
  \be(\bx,\bz):=
  -V^{-1}(\bx,\bz)
\begin{pmatrix}
\sum_{j=k+2l}^n\xi_jx_1^j\\
\vdots\\
\sum_{j=k+2l}^n\xi_jx_k^j\\
\Re\sum_{j=k+2l}^n\xi_jz_1^j\\
\Im\sum_{j=k+2l}^n\xi_jz_1^j\\
\vdots\\
\Re\sum_{j=k+2l}^n\xi_jz_l^j\\
\Im\sum_{j=k+2l}^n\xi_jz_l^j\\
\end{pmatrix}.
\end{equation}
It follows from~\eqref{1400} that~\eqref{2018} is equivalent to
\begin{equation}\label{1404}
\be(\bx,\bz)=
 \begin{pmatrix}
\xi_0\\
\vdots\\
\xi_{k+2l-1}
\end{pmatrix}.
\end{equation}

Consider a random function $\varphi:\R^k\times\C^{l}\to\R$ defined as
\begin{equation}\label{eq-phi-def}
\begin{aligned}
\varphi(\bx,\bz):=\frac1{\rv_{k+2l}(\bx,\bz)}&\prod_{i=1}^{k}\Bigg|\sum_{j=0}^{k+2l-1}j\eta_j(\bx,\bz)x_i^{j-1} + \sum_{j=k+2l}^{n}j\xi_j x_i^{j-1}\Bigg|\\
&\times\prod_{i=1}^{l}\Bigg|\sum_{j=0}^{k+2l-1}j\eta_j(\bx,\bz)z_i^{j-1} + \sum_{j=k+2l}^{n}j\xi_j z_i^{j-1}\Bigg|^2.
\end{aligned}
\end{equation}

\begin{lemma}
  For all $(\bx,\bz)\in\R^k\times\C^{l}$,
\begin{equation}\label{1835}
 \E\left[\varphi(\bx,\bz)\prod_{i=0}^{k+2l-1}f_i(\eta_i(\bx,\bz))\right] =  \rho_{k+2l}(\bx,\bz,\bar\bz).
\end{equation}
\end{lemma}
\begin{proof}
The idea of the proof goes back to~\cite[pp. 58--59]{dK12-Trudy} (see also~\cite[Lemmas~5, 6]{dK14}).

By definition of the expected value,
\begin{align}\label{1850}
\E\left[\varphi(\bx,\bz)\prod_{i=0}^{k+2l-1}f_i(\eta_i(\bx,\bz))\right]=&\frac1{\rv_{k+2l}(\bx,\bz)}\\
\times&\int_{\R^{n-k-2l+1}}\prod_{i=1}^{k}\Bigg|\sum_{j=0}^{k+2l-1}j r_j(\bx,\bz,\bs)x_i^{j-1} + \sum_{j=k+2l}^{n}j s_j x_i^{j-1}\Bigg|\notag\\
\times&\prod_{i=1}^{l}\Bigg|\sum_{j=0}^{k+2l-1}j r_j(\bx,\bz,\bs)z_i^{j-1} + \sum_{j=k+2l}^{n}j s_j z_i^{j-1}\Bigg|^2\notag\\
\times&\prod_{i=0}^{k+2l-1}f_i(r_i(\bx,\bz,\bs))\prod_{i=k+2l}^{n}f_i(s_i)\dd s_{k+2l}\dots \dd s_n,\notag
\end{align}
where the functions $r_0,\dots,r_{k+2l-1}$ are defined by
\begin{equation}\label{1800}
\begin{pmatrix}
r_0(\bx,\bz,\bs)\\
\vdots\\
r_{k+2l-1}(\bx,\bz,\bs)\\
\end{pmatrix}
:=
  -V^{-1}(\bx,\bz)
\begin{pmatrix}
\sum_{j=k+2l}^ns_jx_1^j\\
\vdots\\
\sum_{j=k+2l}^ns_jx_k^j\\
\Re\sum_{j=k+2l}^ns_jz_1^j\\
\Im\sum_{j=k+2l}^ns_jz_1^j\\
\vdots\\
\Re\sum_{j=k+2l}^ns_jz_l^j\\
\Im\sum_{j=k+2l}^ns_jz_l^j\\
\end{pmatrix}
\end{equation}
and $\bs:=(s_{k+2l},\dots,s_n)$.

Now we make the following change of variables:
\begin{equation}\label{1819}
s_i=\sum_{j=0}^{n-k-2l}(-1)^{k+2l-i+j}\sigma_{k+2l-i+j}(\bx,\bz,\bar{\bz})t_j,\quad i=k+2l,\dots,n.
\end{equation}
where $\sigma_i$'s are defined in~\eqref{1738}. The Jacobian is a lower triangle matrix with unities in the diagonal, hence the determinant is 1. This change came from the fact that $x_1,\dots,x_k,z_1,\bar{z_1}\dots,z_l,\bar{z_l}$ are zeros of the polynomial
$$
g(z):=r_0+r_1z+\dots+r_{k+2l-1}z^{k+2l-1}+s_{k+2l}z^{k+2l}+\dots+s_nz^n,
$$
see~\eqref{1800}. Thus for some $t'_0,\dots,t'_{n-k-2l}$ we have
\begin{align}\label{1817}
g(z)&=\prod_{j=1}^{k}(z-x_j)\prod_{j=1}^{l}(z-z_j)(z-\bar{z_j})\left(\sum_{j=0}^{n-k-2l} t'_j z^j \right)\\
&=\left(\sum_{j=0}^{k+2l} (-1)^{k+2l-j} \sigma_{k+2l-j}(\bx,\bz,\bar{\bz}) z^j\right)\left(\sum_{j=0}^{n-k-2l} t'_j z^j \right)\notag
\end{align}
Comparing the coefficients of the polynomials from the left-hand  and right-hand sides and recalling~\eqref{1819} we obtain that $t'_i=t_i$ for $i=0,\dots,n-k-2l$ and
\begin{equation}\label{1846}
r_i(\bx,\bz,\bs)=\sum_{j=0}^{n-k-2l}(-1)^{k+2l-i+j}\sigma_{k+2l-i+j}(\bx,\bz,\bar{\bz})t_j,\quad i=0,\dots,k+2l-1.
\end{equation}
If we differentiate the first equation in~\eqref{1817} at points $x_i,z_i$, and $\bar{z_i}$, we get
\begin{align}\label{1827}
  \sum_{j=0}^{k+2l-1}j r_jx_i^{j-1} + \sum_{j=k+2l}^{n}j s_j x_i^{j-1} &= \prod_{j\ne i}(x_i-x_j)\prod_{j=1}^{l}(x_i-z_j)(x_i-\bar{z_j})\left(\sum_{j=0}^{n-k-2l} t'_j x_i^j \right), \\
\sum_{j=0}^{k+2l-1}j r_jz_i^{j-1} + \sum_{j=k+2l}^{n}j s_j z_i^{j-1} &= \prod_{j=1}^{k}(z_i-x_j)\prod_{j\ne i}(z_i-z_j)(z_i-\bar{z_j})\left(\sum_{j=0}^{n-k-2l} t'_j z_i^j \right),\notag \\
\sum_{j=0}^{k+2l-1}j r_j\bar{z_i}^{j-1} + \sum_{j=k+2l}^{n}j s_j \bar{z_i}^{j-1} &= \prod_{j=1}^{k}(\bar{z_i}-x_j)\prod_{j\ne i}(\bar{z_i}-z_j)(\bar{z_i}-\bar{z_j})\left(\sum_{j=0}^{n-k-2l} t'_j \bar{z_i}^j \right).\notag
\end{align}
Substituting~\eqref{1819}, \eqref{1846}, and~\eqref{1827} in~\eqref{1850} completes the proof of the lemma.

\end{proof}

Denote by $J_{\be}(\bx,\bz)$ the \emph{real} Jacobian matrix of $\be$ at point  $(\bx,\bz)$:
$$
J_{\be}=
\begin{pmatrix}
\frac{\partial \eta_0 }{\partial x_1} & \dots & \frac{\partial \eta_0 }{\partial x_k} & \frac{\partial \eta_0 }{\partial \Re z_1} &  \frac{\partial \eta_0 }{\partial \Im z_1}&\dots & \frac{\partial \eta_0 }{\partial \Re z_l}& \frac{\partial \eta_0 }{\partial \Im z_l}\\
\vdots & \ddots & \vdots & \vdots &  \vdots&\ddots & \vdots&\vdots\\
\frac{\partial \eta_{k+2l-1} }{\partial x_1} & \dots & \frac{\partial \eta_{k+2l-1} }{\partial x_k} & \frac{\partial \eta_{k+2l-1} }{\partial \Re z_1} &  \frac{\partial \eta_{k+2l-1} }{\partial \Im z_1}&\dots & \frac{\partial \eta_{k+2l-1} }{\partial \Re z_l}& \frac{\partial \eta_{k+2l-1} }{\partial \Im z_l}\\
\end{pmatrix}
.
$$

\begin{lemma}\label{1200}
  For all $(\bx,\bz)\in\R^k\times\C^{l},$
\begin{equation}\label{1233}
|\det J_{\be}(\bx,\bz)|
=2^l\varphi(\bx,\bz),
\end{equation}
where $\varphi(\bx,\bz)$ is defined in \eqref{eq-phi-def}.
\end{lemma}
\begin{proof}
Differentiating
\[
V(\bx,\bz)\be(\bx,\bz)=
  -
\begin{pmatrix}
x_1^{k+2l} & x_1^{k+2l+1} & \dots & x_1^{n} \\
\vdots & \vdots & \ddots & \vdots \\
x_k^{k+2l} & x_k^{k+2l+1} & \dots & x_k^{n} \\
\Re z_1^{k+2l} & \Re z_1^{k+2l+1} & \dots & \Re z_1^{n} \\
\Im z_1^{k+2l} & \Im z_1^{k+2l+1} & \dots & \Im z_1^{n} \\
\vdots & \vdots & \ddots & \vdots \\
\Re z_l^{k+2l} & \Re z_l^{k+2l+1} & \dots & \Re z_l^{n} \\
\Im z_l^{k+2l} & \Im z_l^{k+2l+1} & \dots & \Im z_l^{n} \\
\end{pmatrix}
\begin{pmatrix}
\xi_{k+2l}\\
\vdots\\
\xi_n
\end{pmatrix}
,
\]
we obtain
$$
  V(\bx,\bz)J_{\be}(\bx,\bz)+
  \begin{pmatrix}
    A_1 & \mathbf{0} \\
     \mathbf{0} &A_2
  \end{pmatrix}
 =- \begin{pmatrix}
    A_3 & \mathbf{0} \\
     \mathbf{0} & A_4
  \end{pmatrix},
$$
where
$$
A_1:=
\begin{pmatrix}
\sum_{j=0}^{k+2l-1}j\eta_jx_1^{j-1} &0& \dots &0 \\
0&\sum_{j=0}^{k+2l-1}j\eta_jx_2^{j-1} & \dots &0\\
\vdots &\vdots & \ddots & \vdots \\
0&0& \dots & \sum_{j=0}^{k+2l-1}j\eta_jx_k^{j-1} \\
\end{pmatrix},
$$
$$
A_2:=
\begin{pmatrix}
\sum_{j=0}^{k+2l-1}\eta_j\frac{\partial\Re z_1^j}{\partial\Re z_1}  & \sum_{j=0}^{k+2l-1}\eta_j\frac{\partial\Re z_1^j}{\partial\Im z_1}  &\dots & 0&0\\
\sum_{j=0}^{k+2l-1}\eta_j\frac{\partial\Im z_1^j}{\partial\Re z_1}  & \sum_{j=0}^{k+2l-1}\eta_j\frac{\partial\Im z_1^j}{\partial\Im z_1}  &\dots & 0&0\\
 \vdots &  \vdots&\ddots & \vdots&\vdots\\
 0&0 &\dots&\sum_{j=0}^{k+2l-1}\eta_j\frac{\partial\Re z_l^j}{\partial\Re z_l}  & \sum_{j=0}^{k+2l-1}\eta_j\frac{\partial\Re z_l^j}{\partial\Im z_l} \\
  0&0 &\dots&\sum_{j=0}^{k+2l-1}\eta_j\frac{\partial\Im z_l^j}{\partial\Re z_l}  & \sum_{j=0}^{k+2l-1}\eta_j\frac{\partial\Im z_l^j}{\partial\Im z_l} \\
\end{pmatrix},
$$
$$
A_3:=
\begin{pmatrix}
\sum_{j=k+2l}^{n}j\xi_jx_1^{j-1} &0& \dots &0 \\
0&\sum_{j=k+2l}^{n}j\xi_jx_2^{j-1} & \dots &0\\
\vdots &\vdots & \ddots & \vdots \\
0&0& \dots & \sum_{j=k+2l}^{n}j\xi_jx_k^{j-1} \\
\end{pmatrix},
$$
$$
A_4:=
\begin{pmatrix}
\sum_{j=k+2l}^{n}\xi_j\frac{\partial\Re z_1^j}{\partial\Re z_1}  & \sum_{j=k+2l}^{n}\xi_j\frac{\partial\Re z_1^j}{\partial\Im z_1}  &\dots & 0&0\\
\sum_{j=k+2l}^{n}\xi_j\frac{\partial\Im z_1^j}{\partial\Re z_1}  & \sum_{j=k+2l}^{n}\xi_j\frac{\partial\Im z_1^j}{\partial\Im z_1}  &\dots & 0&0\\
 \vdots &  \vdots&\ddots & \vdots&\vdots\\
 0&0 &\dots&\sum_{j=k+2l}^{n}\xi_j\frac{\partial\Re z_l^j}{\partial\Re z_l}  & \sum_{j=k+2l}^{n}\xi_j\frac{\partial\Re z_l^j}{\partial\Im z_l} \\
  0&0 &\dots&\sum_{j=k+2l}^{n}\xi_j\frac{\partial\Im z_l^j}{\partial\Re z_l}  & \sum_{j=k+2l}^{n}\xi_j\frac{\partial\Im z_l^j}{\partial\Im z_l} \\
\end{pmatrix}.
$$

We finish the proof  by taking the second term from the left-hand side to the right-hand side, using~\eqref{1931}, and noting that for any analytic function $f(z)$
$$
 \det\begin{pmatrix}
   \frac{\partial\Re f}{\partial\Re z} &  \frac{\partial\Re f}{\partial\Im z} \\
      \frac{\partial\Im f}{\partial\Re z} &  \frac{\partial\Im f}{\partial\Im z}
  \end{pmatrix}
  =|f'(z)|^2.
$$
\end{proof}

\begin{lemma}[Coarea formula]\label{lm-coarea}
Let $B\subset \R^m$ be a region. Let $\bu:B\to\R^m$ be a Lipschitz function and $h:\R^m\to\R$ be an $L^1$-function. Then
\begin{equation}\label{1225}
\int\limits_{\R^m}\#\{\bx\in B:\bu(\bx)=\by\}\,h(\by)\dd\by=\int\limits_{B}|\det J_{\bu}(\bx)|\, h(\bu(\bx))\dd\bx,
\end{equation}
where $J_{\bu}(\bx)$ is the Jacobian matrix of $\bu(\bx)$.
\end{lemma}
\begin{proof}
See~\cite[pp. 243--244]{hF1969}.
\end{proof}

\subsection{Proof}

Now we are ready to finish the proof of Theorem~\ref{thm-corr-rel}. To this end, we  show that  for any family of mutually disjoint Borel subsets $B_1,\dots,B_k\subset\R$ and $B_{k+1},\dots,B_{k+l}\subset\C_+$,
\begin{equation}\label{1647}
\E\left[\prod_{i=1}^{k+l}\mu(B_i)\right]=2^l\int_{B_1}\dots\int_{B_{k+l}}\,\rho_{k+2l}(\bx,\bz)\dd x_1\dots \dd x_k \dd z_1\dots \dd z_l.
\end{equation}
{From~\eqref{1404} we get
\begin{equation}\label{eq-anoth-way}
\E\,\left[\prod_{i=1}^{k+l}\mu(B_i)\right]
=\E\left[\#\{(\bx,\bz)\in B_1\times\dots\times B_{k+l}:\eta(\bx,\bz)=(\xi_0,\dots,\xi_{k+2l-1})\}\right].
\end{equation}}
For $j=1,\dots,l$ denote by $\tilde B_{k+j}$ a ``real counterpart'' of $B_{k+j}$:
$$
\tilde B_{k+j}:=\{(x,y)\in\R^2:x+iy\in B_{k+j}\}.
$$
Let us apply Lemma~\ref{lm-coarea} {to \eqref{eq-anoth-way}} with
$$
m=k+2l,\quad B=B_1\times\dots\times B_k\times \tilde B_{k+1}\times\dots\times \tilde B_{k+l},
$$
$$
\bu(x_1,\dots,x_{k+2l})=\be(x_1,\dots,x_k,x_{k+1}\pm ix_{k+2},\dots,x_{k+2l-1}\pm ix_{k+2l}),
$$
$$
{h(y_0,\dots,y_{k+2l-1})=f_0(y_0)\dots f_{k+2l-1}(y_{k+2l-1}).}
$$
Note that the indices of $y_i$'s are shifted by 1 according to the polynomial coefficients numeration.
{So due to Lemma~\ref{lm-coarea}, the r.h.s. of \eqref{eq-anoth-way} is equal to}
\begin{multline*}
\int\limits_{\R^{k+2l}}\E\left[\#\{(x_1,\dots,x_{k+2l})\in B\,:\, \bu(x_1,\dots,x_{k+2l})=\by\}\right]f_0(y_0)\dots f_{k+2l-1}(y_{k+2l-1})\,\dd\by\\
=\E\int_B|\det J_{\bu}(\bx)|\,\prod_{i=0}^{k+2l-1}f_i(u_i(x_1,\dots,x_{k+2l}))\,\dd x_1\dots \dd x_{k+2l},
\end{multline*}
where in the last step we used Fubini's theorem and then~\eqref{1225}. The Jacobian matrix of $\bu$ coincides with the real Jacobian matrix of $\be$, the determinant of which is given by Lemma~\ref{1200}. Thus, switching from $\bu$ to $\be$ and again using Fubini's theorem we obtain
\begin{align*}
\E\left[\prod_{i=1}^{k+l}\mu(B_i)\right]&=2^l\int\limits_{B}\E\,\left[\varphi(\bx,\bz)\prod_{i=0}^{k+2l-1}f_i(\eta_i(\bx,\bz))\right]\dd\bx\,\dd\bz,
\end{align*}
{where $\varphi(\bx,\bz)$ is defined in \eqref{eq-phi-def}.}

Combining this with~\eqref{1835} implies~\eqref{1647}, and {due to \eqref{eq-rhokl}} the theorem follows.

%
%

\bibliographystyle{abbrv}
\bibliography{corrf2}

\end{document}